\documentclass{amsart}

\usepackage{amsmath,amssymb,tikz,enumerate,comment,xcolor}
\usetikzlibrary{matrix} 

\numberwithin{equation}{section}
\newtheorem{theorem}[equation]{Theorem}
\newtheorem{lemma}[equation]{Lemma}
\newtheorem{proposition}[equation]{Proposition}
\newtheorem{corollary}[equation]{Corollary}
\newtheorem{definition}[equation]{Definition}

\newcommand{\cat}[1]{\ensuremath{\mathbf{#1}}}
\newcommand{\op}{\ensuremath{{}^{\text{op}}}}
\newcommand{\id}[1][]{\ensuremath{\mathrm{id}_{#1}}}
\newcommand{\ie}{\text{i.e.}~}
\newcommand{\eg}{\text{e.g.}~}

\usepackage{xparse}
\NewDocumentCommand{\xrightarrows}{ O{}O{} }{%
\mathrel{%
\vcenter{\hbox{%
\begin{tikzpicture}
  \node[minimum width=1cm,minimum height=1ex,anchor=south,align=center] (a){\text{\vphantom{hg}#1}\\[0.5ex] \vphantom{hg}#2};
  \draw[<-] ([yshift=0.35ex]a.west) -- ([yshift=0.35ex]a.east);
  \draw[->] ([yshift=-0.35ex]a.west) -- ([yshift=-0.35ex]a.east);
\end{tikzpicture}
}}%
}%
}

\title{Biproducts without pointedness}
\author{Martti Karvonen}

\begin{document}

\begin{abstract}
  We show how to define biproducts up to isomorphism in an arbitrary category without assuming any enrichment. The resulting notion coincides with the usual definitions whenever all binary biproducts exist or the category is suitably enriched, resulting in a modest yet strict generalization otherwise. We also characterize when a category has all binary biproducts in terms of an ambidextrous adjunction. Finally, we give some new examples of biproducts that our definition recognizes.
\end{abstract}

\maketitle

\section{Introduction}

Given two objects $A$ and $B$ living in some category \cat{C}, their biproduct -- according to a standard definition \cite{maclane} -- consists of an object $A\oplus B$ together with maps 

\begin{equation*} A\xrightarrows[$p_A$][$i_A$] A\oplus B\xrightarrows[$i_B$][$p_B$] B \end{equation*}
such that 

\begin{align*}
  p_Ai_A&=\id[A] \quad &p_Bi_B=\id[B] \\
  p_Bi_A&=0_{A,B} \quad &p_Ai_B=0_{B,A} \label{def:biprodswithZeros}
\end{align*}
and
\begin{equation*}\label{def:biprodswithsums}
  \id[A\oplus B]=i_Ap_A+i_Bp_B\text.
\end{equation*}

For us to be able to make sense of the equations, we must assume that \cat{C} is enriched in commutative monoids. One can get a slightly more general definition that only requires zero morphisms but no addition -- that is, enrichment in pointed sets -- by replacing the last equation with the condition that  $(A\oplus B,p_A,p_B)$ is a product of $A$ and $B$ and that $(A\oplus B,i_A,i_B)$ is their coproduct. We will call biproducts in the first sense \emph{additive biproducts} and in the second sense \emph{pointed biproducts} in order to contrast these definitions with our central object of study --  a pointless generalization of biproducts that can be applied in any category \cat{C}, with no assumptions concerning enrichment. This is achieved by replacing the equations referring to zero with the single equation
\begin{equation}\label{eq:projsCommute}
i_Ap_Ai_Bp_B=i_Bp_Bi_Ap_A\text,
\end{equation}
which states that the two canonical idempotents on $A\oplus B$ commute with one another. 

After surveying some basic properties of zero morphisms, we prove that biproducts thus defined behave as one would expect, \eg that they are defined up to unique isomorphism compatible with the biproduct structure, and that the notion agrees with the other definitions whenever \cat{C} is appropriately enriched.  We also show how to characterize them in terms of ambidextrous adjunctions. Both pointed biproducts and the pointless definition studied here add only a modest amount of generality to additive biproducts, for one can show that if \cat{C} has all binary biproducts, then \cat{C} is uniquely enriched in commutative monoids. However, when \cat{C} does not have all biproducts nor zero morphisms, some biproducts recognized by the pointless definition can exist, and we conclude with some examples of this.

Various generalizations of biproducts have been considered before. However, often one assumes a lot of structure from the categories in question, with the goal being a well-behaved notion of an infinite direct sum that is not required to be a product nor a coproduct. Examples of this include~\cite{fritz2019universal,ghezlimaroberts:wstarcategories,wyler1966direct}. In contrast to these, we develop a notion requiring no additional structure on our category while still retaining the universal properties. 

\section{Preliminaries on zeroes}\label{sec:zeroes}

\begin{definition} A morphism $a:A\to B$ is constant if $af=ag$ for all $f,g\colon C\to A$. Coconstant morphisms are  defined dually and a morphism is called a zero morphism if it is both constant and coconstant. A category has zero morphisms if for every pair of objects $A$ and $B$ there is a zero morphism $A\to B$.
\end{definition} 

We recall the definition of a partial zero structure on a category from~\cite{goswami2017structure}. 

\begin{definition} A partial zero structure on a category \cat{C} consists of a non-empty class of morphisms $\mathcal{Z}=\{z_{A,B}\colon A\to B\}$ indexed by some ordered pairs of objects of \cat{C}, subject to the following requirement: for every $z_{A,B}\in\mathcal{Z}$,$f\colon C\to A$ and $g\colon B\to D$, the class $\mathcal{Z}$ also contains a map $z_{C,D}\colon C\to D$ and it equals $g z_{A,B} f$.
\end{definition}

In general, neither zero morphisms between two objects nor partial zero structures on a category are unique -- for instance, in the category $\bullet\rightrightarrows \bullet$ both parallel maps are zero morphisms and form partial zero structures. However, if $A$ has a zero endomorphism, then for any $B$ there is at most one zero map $A\to B$. We list some basic properties of these below.

\begin{proposition}\label{prop:zeroes}
	\begin{enumerate}[(i)]
	\item If $A$ has a zero endomorphism, then for any $B$ there is at most one zero map $A\to B$. In particular, if $A\to B$ is zero and there exists a map $B\to A$, then the zero map $A\to B$ is unique.
	\item A morphism is a zero morphism iff it is part of a partial zero structure. 
	\item The union of partial zero structures $\mathcal{Z}_i$ is a partial zero structure provided that the partial zero structures agree on overlaps, \ie  $z^i_{A,B}=z^{j}_{A,B}$ whenever $z^i_{A,B}\in \mathcal{Z}_i$ and $z^j_i{A,B}$
	\item The class of all morphisms of the form $fzg$ where $z$ is a zero endomorphism forms a partial zero structure provided it is not empty.
	\item A category has zero morphisms iff it is enriched in pointed sets, in which case this enrichment is unique.
	\end{enumerate}
\end{proposition}

\begin{proof} For (i), assume that $f,g\colon A\rightrightarrows  B$ and $h\colon A\to A$ are zero. Then $f=f\id=fh=gh=g\id=g$. The claim after ``in particular'' follows from the fact that if $f\colon A\to B$ is zero and $g\colon B\to A$ is arbitrary, then $gf$ is a zero endomorphism on $A$. 

To prove (ii), assume first that $f$ is a zero morphism. Then morphisms of the form $gfh$ define a partial zero structure. Conversely, if $z_{A,B}$ is part of a zero structure, then $gz_{A,B}f=z_{C,D}=g'z_{A,B}f'$ for all $f,f\rightrightarrows C\to A$ and $g,g'\rightrightarrows B\to D$, showing that  $z_{A,B}$ is a zero morphism.  

(iii) follows straight from the definitions.

Finally, we consider (iv). Note that by (i) this collection picks at most one map $A\to B$ for any $A$ and $B$, whence the claim follows now from (ii) and (iii).

(v) is well-known but also follows readily from (i)-(iv). 
\end{proof}

\section{Main results}\label{sec:main}

We start with the new, enrichment-free definition of a biproduct.

\begin{definition}\label{def:new} A biproduct of $A$ and $B$ in \cat{C} is a tuple $(A\oplus B,p_A,p_B,i_A,i_B)$ such that $(A\oplus B,p_A,p_B)$ is a product of $A$ and $B$, $(A\oplus B,i_A,i_B)$ is their coproduct, and the following equations hold:
\begin{align*}
p_Ai_A&=\id[A] \\
p_Bi_B&=\id[B] \\
i_Ap_Ai_Bp_B&=i_Bp_Bi_Ap_A
\end{align*}
\end{definition}

In the context of Definition~\ref{def:new}, equation~\eqref{eq:projsCommute} could be replaced by alternative equivalent conditions. We list some of them below.

\begin{lemma}\label{lem:absorbing}
Assume that $(A\oplus B,p_A,p_B)$ is a product of $A$ and $B$, $(A\oplus B,i_A,i_B)$ is their coproduct, $p_Ai_A=\id[A]$ and $p_Bi_B=\id[B]$. Then the following are equivalent:
	\begin{enumerate}[(i)]
		\item $(A\oplus B,p_A,p_B,i_A,i_B)$ is the biproduct of $A$ and $B$, \ie the equation \[i_Ap_Ai_Bp_B=i_Bp_Bi_Ap_A\] holds as well
		\item The maps $p_Ai_B$ and $p_Bi_A$ are zero morphisms.
		\item The maps $p_Ai_B$ and $p_Bi_A$ are both constant. 
		\item The map $p_Ai_B$ is coconstant and $p_Bi_A$ is constant.
	\end{enumerate}
\end{lemma}

\begin{proof}
We begin by proving that (i) implies (ii). We first observe that $p_Bi_A$ is coconstant. This follows from the fact that the diagram 
  \[
  \begin{tikzpicture} 
     \matrix (m) [matrix of math nodes,row sep=3em,column sep=4em,minimum width=2em]
     {
     A  &A & A\oplus B & B & C \\
        &  &           &  A\oplus B  & B\\
      A\oplus B & B & A\oplus B &A & A\oplus B\\
       & A & A\oplus B & B\\};
     \path[->]
     (m-3-1) edge node [left] {$p_A$} (m-1-2)
            edge node [below] {$p_B$} (m-3-2)
            edge node [below] {$p_A$} (m-4-2)
     (m-4-2) edge node [below] {$i_A$} (m-4-3)
     (m-4-3) edge node [below] {$p_B$} (m-4-4)
     (m-4-4) edge node [below] {$i_B$} (m-3-5)
     (m-1-1) edge node [above] {$\id[A]$} (m-1-2)
            edge node [left] {$i_A$} (m-3-1)
     (m-3-2) edge node [below] {$i_B$} (m-3-3)
     (m-3-3) edge node [below] {$p_A$} (m-3-4)
     (m-3-4) edge node [below] {$i_A$} (m-3-5)
            edge node [right] {$i_A$} (m-2-4)
     (m-1-4) edge node [left] {$i_B$} (m-2-4)
     (m-2-4) edge node [below] {$h$} (m-1-5)
     (m-1-2) edge node [above] {$i_A$} (m-1-3)
     (m-1-3) edge node [above] {$p_B$} (m-1-4)
     (m-1-4) edge node [above] {$f$} (m-1-5)
     (m-3-5) edge node [right] {$p_B$} (m-2-5)
     (m-2-5) edge node [right] {$g$} (m-1-5); 
  \end{tikzpicture}
  \]
commutes, where $h$ is the cotuple $[g p_Bi_A,f]$. Replacing the roles of $A$ and $B$ proves that $p_Ai_B$ is coconstant as well. Since Definition~\ref{def:new} is self-dual, the maps $p_Ai_B$ and $p_Bi_A$ are also constant and hence zero, establishing that (i) implies (ii).

Condition (ii) clearly implies (iii) and (iv), so it suffices to show that either of them implies (i). Assuming (iii), it suffices to prove that both sides of equation~\eqref{eq:projsCommute} agree when postcomposed with the product projections, and by symmetry it suffices to postcompose only with $p_A$. Doing so to the left hand side results in $p_Ai_Bp_B$ whereas the right hand side yields $p_Ai_Bp_Bi_Ap_A$. These are equal as $p_Ai_B$ is constant. 

Finally, let us assume (iv) and postcompose both sides of equation~\eqref{eq:projsCommute} with the product projections. When postcomposing with $p_B$ we get $p_Bi_Ap_Ai_Bp_B$ and $p_Bi_Ap_A$ which are equal as $p_Bi_A$ is constant. When postcomposing with $p_A$ we are left to show that $p_Ai_Bp_B$ and  $p_Ai_Bp_Bi_Ap_A$ are equal, which we do by precomposing with the coproduct injections. When we precompose with $i_A$ both sides yield $p_Ai_Bp_Bi_A$, whereas precomposing with $i_B$ results in the maps $p_Ai_B$ and $p_Ai_Bp_Bi_Ap_Ai_B$, which are equal since $p_Ai_B$ is coconstant. 
\end{proof}

\begin{corollary}\label{thm:zeros} If \cat{C} has all binary biproducts, then it has zero morphisms.
\end{corollary}
\begin{proof} Combine Lemma~\ref{lem:absorbing} and Proposition~\ref{prop:zeroes}.
\end{proof}

Given Lemma~\ref{lem:absorbing}, it is easy to check that whenever \cat{C} has zero morphisms biproducts and pointed biproducts coincide. Any pointed biproduct is a biproduct in the sense of Definition~\ref{def:new}, since $i_Ap_Ai_Bp_B=0=i_Bp_Bi_Ap_A$. Conversely, let $(A\oplus B,p_A,p_B,i_A,i_B)$ be a biproduct in the sense of Definition~\ref{def:new} in a category with zero morphisms. Now by Lemma~\ref{lem:absorbing} and Proposition~\ref{prop:zeroes} we have $p_Bi_A=0_{A,B}$, as desired, and similarly $p_Ai_B=0_{B,A}$. If \cat{C} is enriched in commutative monoids, then biproducts coincide with additive biproducts just because pointed biproducts and additive biproducts coincide whenever \cat{C} is enriched in commutative monoids.

This shows that whenever all binary biproducts exist, the ordinary definitions suffice just fine. Moreover, every biproduct in the sense of Definition~\ref{def:new} is a pointed biproduct in a suitable subcategory. 

\begin{proposition}\label{prop:biprodsarebiprodsinapointedsubcat} Assume that an object $A$ admits a zero endomorphism. Consider the full subcategory $\cat{C}_0(A)$ of $\cat{C}$ consisting of those objects $B$ that admit a map to and from $A$. Then this subcategory has zero morphisms and any biproduct $A\oplus B$ in the sense of Definition~\ref{def:new} in \cat{C} is a pointed biproduct in $\cat{C}_0(A)$.
\end{proposition}

\begin{proof} As zero morphisms are closed under composition, the category $\cat{C}_0(A)$ has zero morphisms and hence is enriched in pointed sets by~\ref{prop:zeroes}. Moreover $A\oplus B$ is still a biproduct in the sense of Definition~\ref{def:new} as $\cat{C}_{0}(A)$ is full and contains $A,B$ and $A\oplus B$. Hence the discussion preceding this proposition shows that $A\oplus B$ is a pointed biproduct in $\cat{C}_0(A)$.
\end{proof}

Note that assuming a zero endomorphism on $A$ is not a genuine restriction, as it follows from the existence of any biproduct $A\oplus B$.  As pointed biproducts are well-known to be unique up to isomorphism, so are biproducts in the sense of Definition~\ref{def:new}.
 
\begin{corollary} The biproduct of $A$ and $B$, if it exists, is unique up to unique isomorphism compatible with the biproduct structure.
\end{corollary}

Using Proposition~\ref{prop:biprodsarebiprodsinapointedsubcat}, one can then proceed to check that biproducts in our sense work just like one would expect. For example, Definition~\ref{def:new} and other results of this section generalize from the binary case to the biproduct of an arbitrary-sized collection of objects, and one can easily show that if $A\oplus B$ and $(A\oplus B)\oplus C$ exist, then $(A\oplus B)\oplus C$ satisfies the axioms for the ternary biproduct of $A,B,C$. Similarly, one can show that for $f\colon A\to C$ and $g\colon B\to D$ we have $f+g=f\times g$ whenever the biproducts $A\oplus B$ and $C\oplus D$ exist. 

One might take Proposition~\ref{prop:biprodsarebiprodsinapointedsubcat} to mean that no generality is added. While the added generality is relatively modest indeed, some care needs to be taken as the converse of Proposition~\ref{prop:biprodsarebiprodsinapointedsubcat} does not hold in general, \ie a (pointed) biproduct in $\cat{C}_0(A)$ need not be a biproduct in \cat{C}. This is because the inclusion $\cat{C}_0(A)\to \cat{C}$ might fail to preserve products or coproducts. For a deliberately constructed example of this, consider the category $\cat{Vect_k}$ of vector spaces over a field $k$. First add a new initial object $A$ freely to this category, and then add a new morphism $f\colon A\to k\oplus k$ subject to the equation $gf=0$ whenever $g$ is not monic, resulting in a category \cat{C}. Now $0\colon k\oplus k\to k\oplus k$ is still a zero endomorphism and $k\oplus k$ is still a biproduct in $\cat{C}_0(k)\cong \cat{Vect_k}$, but no longer in \cat{C}: the pair $(0_{A,k},0_{A,k})$ factors via $k\oplus k$ both via $0\colon A\to k$ and via $f$.

Recall that an ambiadjoint to a functor $F\colon \cat{C}\to\cat{D}$ is a functor $D\colon \cat{D}\to\cat{C}$ that is simultaneously both left and right adjoint to $F$.

\begin{theorem} \cat{C} has biproducts iff the diagonal $\Delta\colon\cat{C}\to \cat{C}\times \cat C$ has an ambiadjoint $(-)\oplus (-)$ such that the unit $(i_A,i_B)\colon (A,B)\to (A\oplus B,A\oplus B)$ of the adjunction $(-)\oplus (-)\dashv \Delta$, is a section of the counit $(p_A,p_B)\colon (A\oplus B,A\oplus B)\to (A,B)$ of the adjunction $\Delta\dashv (-)\oplus (-)$, \ie $(p_A \circ i_A,p_B\circ i_B)=(\id[A],\id[B])$ for $A,B \in \cat{C}$.
\end{theorem}

\begin{proof}
  The implication from left to right is routine. For the other direction, a right adjoint to the diagonal is well-known to fix binary products, and dually, a left adjoint fixes binary coproducts. Thus it remains to check that the required equations governing $p_A,p_B,i_A$ and $i_B$ are satisfied. By naturality, the diagram
  \[\begin{tikzpicture}[xscale=3.5,yscale=1.5]
    \node (tl) at (0,1) {$A$};
    \node (t) at (1,1) {$A \oplus B$};
    \node (tr) at (2,1) {$B$};
    \node (bl) at (0,0) {$A$};
    \node (b) at (1,0) {$A \oplus C$};
    \node (br) at (2,0) {$C$};
    \draw[->] (tl) to node[above] {$i_A$} (t);
    \draw[->] (t) to node[above] {$p_B$} (tr);
    \draw[->] (bl) to node[below] {$i_A$} (b);
    \draw[->] (b) to node[below] {$p_C$} (br);
    \draw[->] (tl) to node[left] {$\id[A]$} (bl);
    \draw[->] (tr) to node[right] {$f$} (br);
    \draw[->] (t) to node[right] {$\id[A]\oplus f$} (b);
  \end{tikzpicture}\]
  commutes for any $f$. Thus $p_BAi_A$ is coconstant and by duality it is constant, so it is zero. By symmetry this holds also for $p_Ai_B$. As $(p_A \circ i_A,p_B\circ i_B)=(\id[A],\id[B])$ by assumption, the result follows from Lemma~\ref{lem:absorbing}.
\end{proof}

\section{Examples}

Given the results of the previous section, genuinely new examples must be in categories that have neither all binary biproducts nor zero morphisms. One flavor of examples stems from objects that admit few maps in and out of them.

\begin{definition}
An object $A$ of a category is called \emph{subterminal} if for any object $B$ there is at most one morphism $B\to A$.
\end{definition}

\begin{proposition}\label{prop:subterminalgivesbiprods}
If an object $A$ is subterminal both in \cat{C} and in \cat{C}\op, then $(A,\id,\id,\id,\id)$ is the biproduct $A\oplus A$ in \cat{C}.
\end{proposition}

\begin{itemize}
  \item In \cat{Set} (or indeed any topos) the biproduct $\emptyset\oplus\emptyset$ exists and is the empty set.
  \item In any preorder $A\oplus B$ exists if and only if $A\cong B$.
  \item In the category of fields and ring homomorphisms Proposition~\ref{prop:subterminalgivesbiprods} tells us that $F\oplus F$ exists and is isomorphic to $F$ whenever $F$ is a prime field.
\end{itemize}

However, not all novel examples fall under Proposition~\ref{prop:subterminalgivesbiprods}.

\begin{itemize}
  \item Let \cat{C} be any category with biproducts, and let \cat{D} be any non-empty category. Then in the coproduct category $\cat{C}\sqcup \cat{D}$, the biproduct $A\oplus B$ exists whenever $A,B\in\cat{C}$. More concretely, in $\cat{Ab}\sqcup \cat{Set}$ the binary biproduct of any two abelian groups exists and is computed just as in \cat{Ab}, even though $\cat{Ab}\sqcup \cat{Set}$ lacks zero morphisms.
  \item A function $f\colon (X,d_X)\to (Y,d_Y)$ between metric spaces is non-expansive if $d_X(x,y)\geq d_Y(f(x),f(y))$ for all $x,y\in X$. It is contractive if there is some $c\in [0,1)$ such that $c d_X(x,y)\geq d_Y(f(x),f(y))$ for all $x,y\in X$. Let \cat{Met} be the category of metric spaces and non-expansive maps, and let \cat{Con} be the category of contractions. More specifically, let $\mathbb{N}$ denote the monoid of natural numbers. Then \cat{Con} is the full subcategory of $[\mathbb{N},\cat{Met}]$ with objects given by contractive endomorphisms. In \cat{Con}, the terminal object is $! \colon \{*\}\to \{*\}$, and for any $s$ in \cat{Con}, the biproduct $s\oplus !$ exists if and only if $s$ has a (necessarily unique) fixed point.
   \item Let \cat{C} be the category of commutative and cancellative semigroups, that is, of sets equipped with a binary operation $+$ that is associative, commutative, and furthermore satisfies the implication $x+y=x+z\Rightarrow y=z$. Given objects $A$ and $B$ of \cat{C}, the biproduct $A\oplus B$ exists iff either both $A$ and $B$ are empty, or if both $A$ and $B$ have a neutral element, in which case $A\oplus B$ can be constructed just as in \cat{Ab}.
\end{itemize}

\section*{Acknowledgements}

I wish to thank Chris Heunen, Tom Leinster and an anonymous referee for helpful comments. Most of the work was done under the support of Osk. Huttunen foundation, for which I am grateful. 

\bibliographystyle{acm}
\bibliography{biprods}

\end{document}